\documentclass[12pt,reqno]{amsart}

\numberwithin{equation}{section}

\newtheorem{teo}{Theorem }[section]

\newtheorem{lem}[teo]{Lemma}
\newtheorem{prop}[teo]{Proposition}

\newtheorem{rem}{Remark}
\def \e{e^{2bx}}

\begin{document}


\title[KAWAHARA-BURGERS EQUATION]
      { KAWAHARA-BURGERS EQUATION ON A STRIP} 
      
\author[
\ N.~A. Larkin]
{
\ N.~A. Larkin
\bigskip
\\
{\tiny
Departamento de Matem\'atica,\\
Universidade Estadual de Maring\'a,\\
87020-900, Maring\'a - PR, Brazil,\\
 \email{  nlarkine@uem.br}
}
}

 \address
 {
 Departamento de Matem\'atica\\
 Universidade Estadual de Maring\'a\\
 87020-900, Maring\'a - PR, Brazil.
 }
 \email{  nlarkine@uem.br}
 \date{}
\thanks{MSC2010 35Q53;35B35\\
keywords: Kawahara-Burgers equation , Dispersive equations, Exponential
Decay}

\begin{abstract}
An initial-boundary value problem for the  2D
Kawahara-Burgers equation posed  on a channel-type strip
was considered. The existence and uniqueness results for regular and
weak  solutions in weighted spaces as well as exponential decay of
small solutions without restrictions on the width of a strip were
proven both for regular solutions in an elevated norm and for weak
solutions in  the $L^2$-norm.
\end{abstract}

\maketitle

\section{Introduction}\label{introduction}

We are concerned with an initial-boundary value problem (IBVP) for
the two-dimensional Kawahara-Burgers (KB) equation
\begin{equation}\label{kb}
u_t+u_x-u_{xx}+uu_x +u_{xxx}+u_{xyy}-\partial_x^5 u=0\
\end{equation}
posed on a strip modeling an infinite channel
$\{(x,y)\in\mathbb{R}^2:\ x\in \mathbb{R},\,y \in (0,B), \, B>0\}.$
This equation is a two-dimensional analog of the  Kawahara type
 equation
\begin{equation}\label{k}
u_t-\partial_x^5 u +F(u,u_x,u_{xx},u_{xxx})=0
\end{equation}
 which includes dissipation and dispersion and has been
studied intensively last years  due to its applications in Mechanics
and Physics \cite{bona1,benney, biag,bona2, chen,chile,cui,linponce,kwon}.

 Equations \eqref{kb} and \eqref{k} are typical
examples of so-called dispersive equations which attract
considerable attention of both pure and applied mathematicians in
the past decades.
 The theory of the Cauchy problem for (1.2) and other dispersive
equations like the KdV equation has been extensively studied and is
considerably advanced today
\cite{bona1,biag,bona2,chile,chen,cui,kato,kato1,linponce,ponce1,ponce2,kwon,ponce3,saut2}. Results on IBVPs  for one-dimensional dispersive equations both in
bounded and unbounded domains may be found in
\cite{bona2,chile,doronin2,doronin3,kuvsh,larkin2,lar-dor}. It was
shown in \cite{doronin2,doronin3,larkin,larkin1,larkinH1,marcio} that the KdV and Kawahara
equations have an implicit internal dissipation. This allowed the
proof of exponential decay of small solutions in bounded domains
without adding any artificial damping term. Later, this effect has been
proven for a wide class of dispersive equations of any odd order
with one space variable \cite{familark}.
 \par On the other hand, it has been shown in \cite{rozan}
that control of the linear KdV equation with the  transport
term $u_x$  may fail for critical domains, but it is possible to eliminate the term $u_x$ by simple scaling  when the KdV and Kawahara equations are posed on the whole line. The same is true also for \eqref{kb} posed on a strip $(y\in(0,B),\,x\in\mathbb{R},\,t>0 )$ \cite{strip}.

Recently, interest on
dispersive equations became to be extended to  multi-dimensional
models such as Kadomtsev-Petviashvili (KP),\\
 Zakharov-Kuznetsov
(ZK) equations \cite{zk} and dispersive equations of higher orders \cite{phys}. As far as the ZK equation and its
generalizations are concerned, the results on IVPs
 can be found in
\cite{fam1,fara,pastor1,saut,ribaud} and IBVPs
were studied in \cite{fambayk,fam2,fam3,larkin1,h-strip,saut3}.
It was shown that IBVP for the ZK equation posed on
a half-strip unbounded in $x$ direction with the Dirichlet
conditions on the boundaries  possesses regular solutions which
decay exponentially as $t\to \infty$ provided initial data are
sufficiently small and the width of a half-strip is not too large \cite{larkinH1,h-strip}.
The similar result was established for the 2D Kawahara equation posed on a half-strip \cite{larkin1}.
This means that multi-dimensional dispersive equation may create an internal dissipative
mechanism for some types of IBVPs. \par
 The goal of our note is to prove
that the KB equation on a strip also may create a dissipative
effect without adding any artificial damping. We must mention that
IBVP for the ZK equation on a strip $(x\in(0,1),\,y\in\mathbb{R})$
has been studied in \cite{pems,saut3} and IBVPs on a strip
$(y\in(0,L),\,x\in \mathbb{R})$  for the ZK equation and Zakharov-Kuzetsov-Burgers equation were considered
in \cite{fambayk,strip} and for the ZK equation with some internal damping
in \cite{fam3}. In the domain $(y\in(0,B),\,x\in\mathbb{R},\,t>0 )$,
the term $u_x$ in \eqref{kb} can be  scaled out by a simple change
of variables. Nevertheless, it can not be safely ignored for
problems posed  both on finite and semi-infinite intervals as well
as on infinite in $y$ direction bands without changes in the
original domain \cite{pems,rozan}.

The main results of our paper are the existence  and uniqueness of
regular and weak  global-in-time solutions for  \eqref{kb} posed
on a strip with the Dirichlet boundary conditions  and the
exponential decay rate of these solutions as well as continuous
dependence on initial data. To explore dissipativity of the term $u_{xyy},$ we used exponential weight $\e$ which implied to define solutions of \eqref{kb} as the product
$$ e^{bx}[u_t-u_{xx}+uu_x+u_{xxx}+u_{xyy}-\partial_x^5 u]=0\quad \text{in}\quad L^2(\mathcal{S}).$$ 
We must mention that this idea has been proposed yearlier in \cite{kato}.

The paper has the following structure. Section 1 is Introduction.
Section \ref{problem} contains formulation of the problem. In
Section \ref{regexist}, we prove  global existence and uniqueness
theorems for regular solutions in some weighted spaces and
continuous dependence on initial data. Surprisingly, we did not succeed to prove global existence for all
positive weights $e^{2bx}$ as in \cite{larkinH1,h-strip} and  imposed a
restriction $6-40b^2\geq0.$ In Section \ref{regdecay}, we
prove exponential decay of small regular solutions in an elevated
norm. In Section
\ref{weak},
 we prove the existence, uniqueness and continuous dependence on initial data for weak solutions as well
  as the exponential decay rate of the $L^2(\mathcal{S})$-norm for small solutions
without limitations on the width of the strip.

\section{Problem and preliminaries}\label{problem}

Let $B,T,r$ be finite positive numbers. Define
$\mathcal{S}=\{(x,y)\in\mathbb{R}^2:\ x\in\mathbb{R},\ y\in(0,B)\};$
$\mathcal{S}_r=\{(x,y)\in\mathbb{R}^2:\ x\in (-r,+\infty),\,
y\in(0,B)\}$ and $\mathcal{S}_T=\mathcal{S}\times (0,T).$

Hereafter subscripts $u_x,\ u_{xy},$ etc. denote the partial
derivatives, as well as $\partial_x$ or $\partial_{xy}^2$ when it is
convenient. Operators $\nabla$ and $\Delta$ are the gradient and
Laplacian acting over $\mathcal{S}.$ By $(\cdot,\cdot)$ and
$\|\cdot\|$ we denote the inner product and the norm in
$L^2(\mathcal{S}),$ and $\|\cdot\|_{H^k}$ stands for norms in the
$L^2$-based Sobolev spaces. We will use also the spaces $H^s\cap
L^2_b$, where $L^2_b=L^2(e^{2bx}dx)$, see \cite{kato}.

Consider the following  IBVP:
\begin{align}
&Lu\equiv u_t-u_{xx}+uu_x+u_{xxx}+u_{xyy}-\partial^5_x u=0,\ \ \text{in}\
\mathcal{S}_T; \label{2.1}
\\
&u(x,0,t)= u(x,B,t)=0,\; x\in \mathbb{R},\ t>0; \label{2.2}
\\
&u(x,y,0)=u_0(x,y),\ \ (x,y)\in\mathcal{S}. \label{2.3}
\end{align}

\section{Existence of regular solutions}\label{regexist}

{\bf Approximate solutions}. We will construct  solutions to
\eqref{2.1}-\eqref{2.3} by the Faedo-Galerkin method: let ${w_j(y)}$
be orthonormal in $L^2(\mathcal{S})$ eigenfunctions of the following
Dirichlet problem:
\begin{align}
&w_{jyy}+\lambda_j w_j=0, \,y\in (0,B); \label{2.4} \\
&w_j(0)=w_j(B)=0. \label{2.5}
\end{align}

Define approximate solutions of \eqref{2.1}-\eqref{2.3} as follows:
\begin{equation}
 u^N(x,y,t)=\sum^N_{j=1} w_j(y)g_j(x,t), \label{UN}
\end{equation}
where $g_j(x,t)$ are solutions to the following Cauchy problem for
the system of $N$ generalized Kawahara equations:
\begin{align}
&\frac{\partial}{\partial t}g_j(x,t)+\frac{\partial^3}{\partial
x^3}g_j(x,t)-\frac{\partial^2}{\partial x^2}g_j(x,t)-\lambda_j\frac{\partial}{\partial x} g_{j}(x,t) \nonumber\\
&-\frac{\partial^5}{\partial x^5} g_j(x,t)+\int^B_0 u^N(x,y,t)u^N_x(x,y,t)w_j(y)\, dy
=0, \label{2.6}\\
&g_j(x,0)=\int^B_0 w_j(y)u_0(x,y)\,dy,\;j=1,...,N. \label{2.7}
\end{align}
It can be shown that for $g_j(x,0)\in H^s,\,s\geq 5,$ the Cauchy problem
\eqref{2.6}-\eqref{2.7} has a unique regular solution $g_j\in
L^{\infty}(0,T;H^s(\mathcal{S})\cap L^2_b(\mathcal{S}))\cap
L^2(0,T;H^{s+2}(\mathcal{S})\cap L^2_b(\mathcal{S}))$
\cite{bona1,linponce,kato,ponce3}.
 To prove the existence of  global solutions for
 \eqref{2.1}-\eqref{2.3}, we need uniform in $N$  global in $t$
 estimates of approximate solutions $u^N(x,y,t).$\\
 {\bf  Estimate I.} Multiply the j-th equation of \eqref{2.6} by
 $g_j$, sum up over $j=1,...,N$ and integrate the result with respect
 to $x$ over $\mathbb{R}$ to obtain
 $$\frac{d}{dt}\|u^N\|^2(t)+2\|u^N_x\|^2(t)= 0\nonumber$$
 which implies
 \begin{equation}
\|u^N\|^2(t) +2\int_0^t\|u^N_x\|^2(s)\,ds=\|u^N_0\|^2\quad \forall t
\in (0,T). \label{FE}
\end{equation}

It follows from here that for $N$ sufficiently large and $\forall
t>0$

\begin{equation}
\|u^N\|^2(t)+2\int^t_0\ \|u^N_{x}\|^2(s)\,ds =\|u^N\|^2(0)\leq
2\|u_0\|^2. \label{E1}
\end{equation}
In our calculations we will drop the index $N$ where it is not
ambiguous.

{\bf  Estimate II.} For some positive $b$, multiply the j-th
equation of \eqref{2.6} by $e^{2bx}g_j$ , sum up over $j=1,...,N$
and integrate the result with respect  to $x$ over $\mathbb{R}.$
Dropping the index $N$, we get
\begin{align}
&\frac{d}{dt}(e^{2bx},u^2)(t)+(2+6b-40b^3)(e^{2bx},u^2_x)(t)+2b(e^{2bx},u^2_y)(t)\nonumber\\
&+10b(\e,u_{xx}^2)(t)-\frac{4b}{3}(e^{2bx},u^3)(t)\nonumber\\ &-(4b^2+8b^3-32b^5)(e^{2bx},u^2)(t)=0.
\label{2e}
\end{align}
\begin{prop}
Let $b\in(0,\frac{\sqrt{0,6¨}}{2}]$, then
\begin{equation} \label{bcond}
6b-40b^3\geq 0.
\end{equation}
\end{prop}
The proof is obvious.

In our calculations, we will frequently use the following
multiplicative inequalities \cite{lady2}:
\begin{prop} \label{GN}
i) For all $u \in H^1(\mathbb{R}^2)$
\begin{equation} {\|u\|}_{L^4(\mathbb{R}^2)}^2 \leq 2 {\|u\|}_{L^2(\mathbb{R}^2)}{\|\nabla u\|}_{L^2(\mathbb{R}^2)}.
 \label{p1}
\end{equation}
\qquad \qquad \qquad \qquad ii) For all $u \in H^1(D)$
\begin{equation} {\|u\|}_{L^4(D)}^2 \leq C_D {\|u\|}_{L^2(D)}{\|u\|}_{H^1(D)}, \label{p2}
\end{equation}
where the constant $C_D$ depends on a way of continuation of $u \in
 H^1(D)$ as $ \tilde{u}(\mathbb{R}^2)$ such that $\tilde{u}(D)=u(D).$
\end{prop}

Extending $u^N(x,y,t)$ for a fixed $t$ into the exterior of
$\mathcal{S}$ by 0 and exploiting 
\eqref{p1}, we find
 \begin{equation}
\frac{4b}{3}(e^{2bx}u^3)(t)\leq
b(e^{2bx},u^2_y)(t)\nonumber\\+2b(e^{2bx},u^2_x)(t)+2(b^3+\frac{8b}{9}\|u_0^N\|^2)(e^{2bx},u^2)(t).\label{nl}
\end{equation}
Substituting this into \eqref{2e}, we come to the inequality
\begin{align}
&\frac{d}{dt}(e^{2bx},u^2)(t)+(e^{2bx},u^2_x+u^2_y+u_{xx}^2)(t)\nonumber\\
&\leq C(b)(1+\|u_0\|^2)(e^{2bx},u^2)(t). \label{2.9}
\end{align}
By the Gronwall lemma,
$$(e^{2bx},u^2)(t)\leq C(b,T,\|u_o\|)(e^{2bx},u^2_0).\nonumber$$
Returning to \eqref{2.9} gives
\begin{align}
(e^{2bx},|u^N|^2)(t)+\int_0^t (e^{2bx},|\nabla u^N|^2+|u^N_{xx}|^2)(\tau )d
\tau\nonumber\\ \leq C(b,T,\|u_0\|)(e^{2bx},u^2_0)\quad \forall t
\in (0,T). \label{E2}
\end{align}
It follows from this estimate and (3.6) that uniformly in $N$ and
for any $r>0$ and $t \in (0,T)$
\begin{align}
 &\|u^N\|^2(t)+\int_0^t\int_0^B\int_{-r}^{+\infty}
[|\nabla u^N|^2+|u^N_{xx}|^2]\,dx\,dy\,ds\nonumber\\& \leq
\mathbb{C}(r,b,T,\|u_0\|)(e^{2bx},u_0^2), \label{H1}
\end{align}
 where
$\mathbb{C}$ does not depend on $N$.

Estimates \eqref{E2}, \eqref{H1} make it possible to prove the
existence of a weak solution to \eqref{2.1}-\eqref{2.3} passing to
the limit in \eqref{2.6} as $N \to \infty$. For details of passing
to the limit  in the nonlinear term  see \cite{kato}.\par
We will need the following lemma :
\begin{lem} \label{supr}
Let $u(x,y): \mathcal{S}\to \mathbb{R}$ be such that
$$ \int_{\mathcal{S}}\e [u^2(x,y)+|\nabla
u(x,y)|^2+u_{xy}^2(x,y)]\,dxdy < \infty$$ and for all
$x\in\mathbb{R}$ there is some $y_0\in [0,B]$ such that
$u(x,y_0)=0.$ Then
\begin{align}
&\sup_{\mathcal{S}}|e^{bx}u(x,y,t)|^2 \leq
\delta(1+2b^2)(\e,u_y^2)(t)+
2\delta(\e,u_{xy}^2)(t)\nonumber\\
&+\frac{2\delta_1}{\delta}(\e,
u_x^2)(t)+\frac{1}{\delta}\big[\frac{1}{\delta_1}+2\delta_1
b^2\big](\e,u^2)(t),\label{esup}
\end{align}
where $\delta, \delta_1$ are arbitrary positive numbers.
\end{lem}

\begin{proof} Denote $v=e^{bx}u.$ Then simple calculations give
\begin{align}
&\sup_{\mathcal{S}} v^2(x,y,t)\leq
\delta[\|v_y\|^2(t)+\|v_{xy}\|^2(t)] +\frac{1}{\delta}[\|v_x\|^2(t)+
\|v\|^2(t)].\nonumber
\end{align}
Returning to the function $u(x,y,t)$, we prove Lemma \ref{supr}
\end{proof}

{\bf  Estimate III.} Multiplying the j-th equation of \eqref{2.6} by
$-(\e g_{jx})_x$, and dropping the index $N$, we come to the
equality
\begin{align}
&\frac{d}{dt}(\e,
u_x^2)(t)+(2+6b-40b^3)(\e,u_{xx}^2)(t)+2b(\e,u_{xy}^2)(t)\nonumber\\
&+10b(\e,u_{xxx}^2)(t)-(4b^2+8b^3-32b^5)(\e,u_x^2)(t)\nonumber\\
&+(\e,u_x^3)(t)-2b(\e,uu_x^2)(t)=0.\label{e3}
\end{align}
Making use of Proposition \ref{GN}, we estimate
\begin{align}
&I_1=(\e,u_x^3)(t)\leq
\|u_x\|(t)\|e^{bx}u_x\|^2_{L^4(\mathcal{S})}(t\nonumber\\
&\leq 2\|u_x\|(t)\|e^{bx}u_x\|(t)\|\nabla(e^{bx} u_x)\|(t)\nonumber\\
&\leq \delta(\e,2u_{xx}^2+u_{xy}^2)(t)+2\big[\delta
b^2+\frac{\|u_x\|^2(t)}{2\delta}\big](\e,u_x^2)(t).\nonumber
\end{align}
Similarly,
\begin{align}
&I_2=2b(\e,uu_x^2)(t)\leq \delta(\e,2u_{xx}^2+u_{xy}^2)(t)\nonumber\\
&+\big[2b^2\delta+\frac{4b^2}{\delta}\|u_0\|^2(t)\big](\e,u_x^2)(t).\nonumber
\end{align}
Substituting $I_1,I_2$ into \eqref{e3} with $2\delta=b$, we
obtain for $\forall t\in(0,T):$
\begin{align}
&(\e,|u^N_x|^2)(t)+\int_0^t(\e,|\nabla
u^N_x|^2+|u^N_{xxx}|^2 )(s)\,ds\nonumber\\&\leq C(b,T,\|u_0\|)(\e,u_{0x}^2).
\label{E3}
\end{align}

{\bf  Estimate IV.} Multiplying the j-th equation of \eqref{2.6} by
$-2(\e\lambda_j g_{j})$, and dropping the index $N$, we come to the
equality
\begin{align}
&\frac{d}{dt}(\e,
u_y^2)(t)+(2+6b-40b^3)(\e,u_{xy}^2)(t)+2b(\e,u_{yy}^2)(t)\nonumber\\
&+10b(\e,u_{xxy}^2)(t)-(4b^2+8b^3-32b^5)(\e,u_y^2)(t)\nonumber\\
&+2(1-b)(\e,u_x u_y^2)(t)=0.\label{e4}
\end{align}
Making use of Proposition \ref{GN}, we estimate
\begin{align}
&I=2(1-b)(\e,u_x u_y^2)(t)\nonumber\\&\leq
2C_D(1+b)\|u_x\|(t)\|e^{bx}u_y\|(t)\|
(e^{bx}u_y)\|_{H^1(\mathcal{S})}(t)\nonumber\\
&\leq\delta(\e, 2u_{xy}^2+u_{yy}^2)(t)+\big[2\delta
(1+b^2)\nonumber\\&+\frac{C_D^2(1+b)^2\|u_x\|^2(t)}{\delta}\big](\e,u_y^2)(t).\nonumber
\end{align}
Taking $\delta=b,$ we transform \eqref{e4} into the inequality
\begin{align}
&\frac{d}{dt}(\e,
u_y^2)(t)+(\e,u_{xy}^2+u_{yy}^2+u_{xxy}^2)(t)\nonumber\\
&\leq C(b)[1+\|u_x\|(t)^2](\e,u_y^2)(t).\nonumber
\end{align}
Making use of \eqref{E1} and the Gronwall lemma, we get $\;\forall
t\in(0,T):$
\begin{align}
&(\e,|u^N_y|^2)(t)+\int_0^t(\e,|u^N_{yy}|^2+|u^N_{xy}|^2+|u^N_{xxy}|^2)(s)\,ds\nonumber\\\leq
&C(b,T,\|u_0\|)(\e,u_{0y}^2). \label{E4}
\end{align}

This and \eqref{E3} give for $\forall t\in( 0,T)$:
\begin{align}
&(\e,|\nabla u^N|^2)(t)+\int_0^t(\e,|\nabla u^N_x|^2+|\nabla u^N_{xx}|^2+| u^N_{yy}|^2)(s)\,ds\nonumber\\
&\leq C(b,T,\|u_0\|)|_0)(\e,|\nabla u_0|^2)\label{E5}
\end{align}

which imply that for all finite $r>0$ and all
$t\in(0,T)$
\begin{equation}
\|u^N\|(t)_{H^1(\mathcal{S}_r)}\leq C(r,b,T,\|u_0\|)(\e,|\nabla
u_0|^2). \label{Ert}
\end{equation}

{\bf  Estimate V.} Multiplying the j-th equation of \eqref{2.6} by
$(\e g_{jxx})_{xx}$, and dropping the index $N$, we come to the
equality
\begin{align}
&\frac{d}{dt}(\e,
u_{xx}^2)(t)+(2+6b-40b^3)(\e,u_{xxx}^2)(t)+2b(\e,u_{xxy}^2)(t)\nonumber\\
&+10b(\e,u^2_{xxxx})(t)-(4b^2+8b^3-32b^5)(\e,u_{xx}^2)(t)\nonumber\\&-2b(\e,u u_{xx}^2)(t)+5(\e
u_x,u_{xx}^2)(t)=0.\label{e5}
\end{align}

Using \eqref{p1}, we estimate
\begin{align}
&I=-2b(\e,u u_{xx}^2)(t)+5(\e u_x,u_{xx}^2)(t)\nonumber\\
&\leq
2\delta(\e,2u_{xxx}^2+u_{xxy}^2)(t)+\big[4b^2\delta+\frac{25}{\delta}\|u_x\|(t)^2\nonumber\\
&+\frac{4b^2}{\delta}\|u\|^2(t)\big](\e,u_{xx}^2)(t).\nonumber
\end{align}
Taking $2\delta=b$ and substituting $I$ into \eqref{e5}, we obtain
\begin{align}
&\frac{d}{dt}(\e,
u_{xx}^2)(t)+(\e,u_{xxx}^2+u_{xxy}^2+u_{xxxx}^2)(t)\nonumber\\
&\leq C(b)[1+\|u_x\|^2(t)+\|u\|^2(t)](\e,u_{xx}^2)(t).\nonumber
\end{align}
Making use of \eqref{E1}, we find
\begin{align}
&(\e,|u^N_{xx}|^2)(t)+\int_0^t(\e,|\nabla
u^N_{xx}|^2+|u^N_{xxxx}|^2)(s)\,ds\nonumber\\\leq
&C(b,T,\|u_0\|)(\e,u_{0xx}^2)\quad \forall t\in(0,T).\label{E5}
\end{align}

{\bf  Estimate VI.} Differentiate \eqref{2.6} by $t$ and multiply
the result by $\e g_{jt}$ to obtain

\begin{align}
&\frac{d}{dt}(\e,
u_t^2)(t)+(2+6b-40b^3)(\e,u_{xt}^2)(t)+2b(\e,u_{ty}^2)(t)\nonumber\\
&+10b(\e,u_{txx}^2)(t)-(4b^2+8b^3-32b^5)(\e,u_{t}^2)(t)\nonumber\\&+(2-2b)(\e u_x, u_{t}^2)(t)=0.\label{e6}
\end{align}
 Making use of \eqref{p1}, we estimate
 \begin{align}
 &I=(2-2b)(\e u_x, u_{t}^2)(t)\leq
 2(2+2b)\|u_x\|(t)\|e^{bx}u_t\|(t)\|\nabla(e^{bx}
 u_t)\|(t)\nonumber\\
 &\leq\delta(\e,2u_{xt}^2+u_{ty}^2)(t)+\big[2b^2\delta+\frac{(2+2b)^2\|u_x\|(t)^2}{\delta}\big](\e,u_t^2)(t).\nonumber
 \end{align}
 Taking $\delta=b$ and substituting $I$ into \eqref{e6}, we get

\begin{align}
&\frac{d}{dt}(\e,
u_t^2)(t)+(\e,u_{xt}^2+u_{ty}^2+u_{txx}^2)(t)\nonumber\\
&\leq C(b)[1+\|u_x\|(t)^2](\e , u_{t}^2)(t).\nonumber
\end{align}

This implies \quad $\forall t\in 0,T)$:
\begin{align}
(\e,|u^N_t|^2)(t)+\int_0^t(\e,|\nabla
u^N_{s}|^2+|u^N_{sxx}|^2)(s)\,ds\nonumber\\\leq C(b,T,\|u_0\|)(\e,u_{t}^2)(0)\leq
C(b,T,\|u_0\|)\|)J_0,\label{E6}
\end{align}
where
$$J_0=\|u_0\|^2+(\e,u^2_0+|\nabla u_0|^2+|\nabla u_{0x}|^2+u^2_0
u^2_{0x}+|\Delta u_{0x}|^2+|\partial_x^5 u_0|^2).$$

{\bf  Estimate VII.} Multiplying the j-th equation of \eqref{2.6} by
$-\e g_{jx}$ and dropping the index $N$, we come to the equality
\begin{align}
&(\e,[u_{xy}^2+u_{xxx}^2])(t)=(\e, uu_x^2)(t)+(\e u_{t},u_x)(t)\nonumber\\&+(8b^2-1)(\e,u_{xx}^2)(t)+(b+2b^2-8b^4)(\e,u^2_x)(t).
\label{e7}
\end{align}
Using \eqref{p1}, we estimate
\begin{align}
&I=(\e,uu_x^2)(t)\leq \delta(\e,2u_{xx}^2+u_{xy}^2)(t)+
\big[2b^2\delta+\frac{\|u_0\|^2}{\delta}\big](\e,u_x^2)(t).\nonumber
\end{align}
Taking $2\delta=1$, using \eqref{E3}-\eqref{E6} and substituting $I$
into \eqref{e7}, we get

\begin{equation}
(\e,{u^N_{xxx}}^2+{u^N_{xy}}^2)(t)\leq C(b,T,\|u_0\|)J_0 \quad
\forall t\in(0,T). \label{E7}
\end{equation}

{\bf  Estimate VIII.}

Multiplying the j-th equation of \eqref{2.6} by $\e g_{jxxx}$, we
come, dropping the index $N$, to the equality
\begin{align}
&(\e,
u_{xxy}^2+u_{xxxx}^2)(t)=-(\e[u_t-u_{xx}],u_{xxx})(t)-(\e
uu_x,u_{xxx})(t)\nonumber\\&+2b^2(\e,u_{xy}^2)(t)+(2b^2-1)(\e,u_{xxx}^2)(t).\label{e8}
\end{align}

Using Lemma \ref{supr} and \eqref{E1}, we estimate
\begin{align}
&I=(\e uu_x,u_{xxx})(t)\leq
\|u\|(t)\sup_{\mathcal{S}}|e^{bx}u_x(x,y,t)|\|e^{bx}u_{xxx}\|(t)\nonumber\\
&\leq
\frac{\|u_0\|^2}{2}(\e,u_{xxx}^2)(t)+\frac{1}{2}\big[\frac{1}{\delta}(1+2b^2)(\e,u_x^2)(t)\nonumber\\
&+\frac{2}{\delta}(\e
u_{xx}^2)(t)+\delta(1+2b^2)(\e,u_{xy}^2)(t)+2\delta(\e,u_{xxy}^2)(t)\big].
\end{align}

Taking  $\delta$ sufficiently small, positive and
substituting $I$ into \eqref{e8}, we find
\begin{equation}
(\e, |\nabla u^N_{xx}|^2+|\partial^4_x u^N|^2)(t)\leq C(b,T,\|u_0\|)J_0 \quad \forall
t\in(0,T). \label{E81}
\end{equation}

Consequently, it follows from the equalities:
$$ -(\e [u^N_t-u^N_{xx}+u^N_{xxx}+u^N_{xyy}+u^N
u^N_x-\partial^5_x u^N],u^N_{yy})(t)=0,
$$

$$ -(\e [u^N_t-u^N_{xx}+u^N_{xxx}+u^N_{xyy}+u^N
u^N_x-\partial^5_x u^N],\partial^5_x u^N)(t)=0
$$

and 

$$ (\e [u^N_t-u^N_{xx}+u^N_{xxx}+u^N_{xyy}+u^N
u^N_x-\partial^5_x u^N],u^N_{xyy})(t)=0
$$
 that
\begin{align}
&(\e, |u^N_{yy}|^2+|u^N_{xyy}|^2+|\partial^5_x u^N|^2+|u^N_{xxxy}|^2)(t)\nonumber\\&\leq C(b,T,\|u_0\|)J_0 \quad
\forall t\in(0,T). \label{E82}
\end{align}

Jointly, estimates \eqref{E3},\eqref{E4}, \eqref{E5},
\eqref{E7},\eqref{E81}, \eqref{E82} read
\begin{align}
&(\e, |u^N|^2+|\nabla u^N|^2 +|\nabla u^N_x|^2+|\nabla
u^N_y|^2+|\nabla u_{xx}^N|^2+|\Delta u^N_x|^2\nonumber\\&+|\nabla u^N_{xxx}|^2+|\partial_x^5 u^N|^2)(t)\leq
C(b,T,\|u_0\|)J_0\quad \forall t\in(0,T). \label{E83}
\end{align}
In other words,
\begin{align}
&e^{bx}u^N,\quad e^{bx}u^N_x\in L^{\infty}(0,T;H^2(\mathcal{S}))\nonumber\\
& \nabla u^N_{xxx},\quad \partial^5_x u^N \in L^{\infty}(0,T;L^2(\mathcal{S})) 
\label{E84}
\end{align}
and these inclusions are uniform in $N$.

{\bf  Estimate IX.}
 Multiplying the j-th equation of \eqref{2.6} by
$\e\lambda^2_j g_{j}$, we come, dropping the index $N$, to the equality
\begin{align}
&b(\e,5
u_{xxyy}^2+u_{yyy}^2)(t)=(2b^2+4b^3-16¨b^5)(\e,u_{yy}^2)(t)\nonumber\\&+(20b^3-3b-1)(\e
,u^2_{xyy})(t)+(\e,u_{ty},u_{yyy})(t)+(\e uu_{xy},u_{yyy})(t)\nonumber\\&+(\e u_y u_x,u_{yyy})(t).\label{e10}
                          \end{align}
We estimate
\begin{align}
& I_1=-(\e,u_{ty},u_{yyy})(t)\leq
\frac{\epsilon}{2}(\e,u_{yyy}^2)(t)+\frac{1}{2\epsilon}(\e,u_{yt}^2)(t),\nonumber\\
&I_2=(\e u_y
u_x,u_{yyy})(t)\leq\|u_x\|(t)\|e^{bx}u_{xyyy}\|(t)\sup_{\mathcal{S}}|e^{bx}u_y(x,y,t|\nonumber\\
&\leq
\frac{\epsilon}{2}(\e,u_{xyyy}^2)(t)+\frac{\|u_x\|(t)^2}{2\epsilon}\big[(1+2b^2)(\e,u_y^2)(t)\nonumber\\
&+2(\e,u_{xy}^2)(t)+(1+2b^2)(\e,u_{yy}^2)(t)
+2(\e,u_{xyy}^2)(t)\big],\nonumber\\
&I_3=(\e
uu_{xy},u_{yyy})(t)\leq\|u\|(t)\|e^{bx}u_{yyy})\|(t)\sup_{\mathcal{S}}|e^{bx}u_{xy}(x,y,t|\nonumber\\
&\leq
\frac{\|u_0\|^2\epsilon_1}{2}(\e,u_{yyy}^2)(t)+\frac{1}{2\epsilon_1}\big[ 2\delta(\e, u_{xxyy}^2)(t)\nonumber\\
&+\frac{2}{\delta}(\e,u_{xxy}^2)(t)+\delta(1+2b^2)(\e,u_{xyy}^2)(t)\nonumber\\
&+\frac{1}{\delta}(1+2b^2)(\e,u_{xy}^2)(t)\big].\nonumber
\end{align}
Choosing $\epsilon,\;\epsilon_1,\; \delta$ sufficiently small,
positive, after integration, we transform \eqref{e10}  into the form
\begin{equation}
\int_0^T(\e,[|u^N_{xxyy}|^2+|u^N_{yyy}|^2])(t)\,dt\leq
C(b,T,\|u_0\|)J_0. \label{E101}
\end{equation}
Acting similarly, we get from the scalar product
$$(\e\big[u^N_t-u^N_{xx}+u^N_{xxx}+u^N_{xyy}+u^N u^N_x-\partial_x^5 u^N\big],u^N_{xyyyy})(t)=0$$
the estimate
\begin{equation}
\int_0^T(\e,|u^N_{xyyy}|^2+|\partial_x^3 u^N_{yy}|^2)(t)\,dt \leq C(b,T,\|u_0\|)J_0.
\label{E102}
\end{equation}

 Estimates \eqref{E83}, \eqref{E84},  \eqref{E101},
\eqref{E102} guarantee that
\begin{equation}
e^{bx}u^N,\quad e^{bx}u^N_x\in L^{\infty}(0,T;H^2(\mathcal{S})\cap
L^2(0,T;H^3(\mathcal{S}))\label{EN}
\end{equation}
and these inclusions do not depend on $N.$ Independence of Estimates
\eqref{E1},\eqref{EN} of $N$ allow us to pass to the limit in
\eqref{2.6} and to prove the following result:
\begin{teo} \label{regsol}
Let $u_0(x,y):\mathbb{R}^2\to \mathbb{R}$ be such that
$u_0(x,0)=u_0(x,B)=0$ and for some $b>0$ satisfying \eqref{bcond}

$$J_0= \|u_0\|^2+(\e,u_0^2+|\nabla u_0|^2+|\nabla
u_{0x}|^2+u_0^2 u_{0x}^2 + |\Delta u_{0x}|^2+|\partial_x^5 u_0|^2) < \infty.$$
 Then there exists a regular solution to
\eqref{2.1}-\eqref{2.3} $u(x,y,t):$
\begin{align}
& u\in L^{\infty}(0,T;L^2(\mathcal{S})),\quad u_x\in
 L^2(0,T;L^2(\mathcal{S}))\nonumber\\
 &e^{bx}u,\;e^{bx}u_x \in L^{\infty}(0,T;H^2(\mathcal{S}))\cap
 L^2(0,T;H^3(\mathcal{S}))\nonumber\\
 &e^{bx}u_t \in L^{\infty}(0,T;(L^2(\mathcal{S})))\cap
 L^2(0,T;H^1(\mathcal{S})),\nonumber\\
 &e^{bx}u_{xxt}\in L^2(0,T;L^2(\mathcal{S})),\quad e^{bx}\partial_x^5 u\in L^2(0,T;H^1(\mathcal{S}))\nonumber
 \end{align}
 which for $a.e. \;t\in(0,T)$ satisfies the identity
\begin{equation}
(e^{bx}\big[u_t-u_{xx}+u_{xxx}+uu_x+u_{xyy}-\partial_x^5 u\big],\phi(x,y))(t)=0,
\label{regularsol}
\end{equation}
where $\phi(x,y)$ is an arbitrary function from $L^2(\mathcal{S}).$
\end{teo}
\begin{proof}
Rewrite \eqref{2.6} in the form
\begin{align}
&(e^{bx}\big[u^N_t-u^N_{xx}+u^N
u^N_x+u^N_{xxx}+u^N_{xyy}\nonumber\\&-\partial_x^5 u^N\big],\Phi^N(y)\Psi(x))(t)=0, \label{EqN}
\end{align}
where $\Phi^N(y)$ is an arbitrary function from the set of linear
combinations $\sum_{i=1}^N \alpha_i w_i(y)$ and $\Psi(x)$ is an
arbitrary function from $H^1(\mathbb{R})$. Taking into account
estimates \eqref{E1}, \eqref{EN} and fixing $\Phi^N$, we can easily
pass to the limit as $N\to\infty$ in linear terms of \eqref{EqN}. To
pass to the limit in the nonlinear term, we must use \eqref{Ert} and
repeat arguments of \cite{kato}. Since linear combinations
$[\sum_{i=1}^N \alpha_i w_i(y)]\Psi(x)$ are dense in
$L^2(\mathcal{S}),$ we come to \eqref{regularsol}. This proves the
existence of regular solutions to (2.1)-(2.3).
\end{proof}

\begin{rem} \label{limit estimates}
Estimates \eqref{E1},\eqref{EN} are valid also for  the limit
function $u(x,y,t)$ and \eqref{E1} obtains its sharp form:
\begin{equation}
\|u\|(t)^2+2\int_0^t \|u_x\|^2(s)\,ds=\|u_0\|^2 \quad \forall
t\in(0,T). \label{E1l}
\end{equation}
\end{rem}

{\bf  Uniqueness of a regular solution.}
\begin{teo} \label{uniq} A
regular solution from Theorem \ref{regsol} is uniquely defined.
\end{teo}
\begin{proof}
Let $u_1,\,u_2$ be two distinct regular solutions of
\eqref{2.1}-\eqref{2.3}, then $z=u_1-u_2$ satisfies the following
initial-boundary value problem:
\begin{align}
&z_t-z_{xx}+z_{xxx}+z_{xyy}-\partial_x^5 z +\frac{1}{2}(u_1^2-u_2^2)_x=0
\;\mbox{in} \;\mathcal{S}_T,
\label{u1}\\
&z(x,0,t)=z(x,B,t)=0,\quad x\in \mathbb{R},\quad t>0,\label{u2}\\
&z(x,y,0)=0.\quad (x,y)\in \mathcal{S}. \label{u3}
\end{align}
Multiplying \eqref{u1} by $2e^{bx}z$, we get
\begin{align}
&\frac{d}{dt}(\e,z^2)(t)+(2+6b-40b^3)(\e,z_x^2)(t)+10b(\e,z_{xx}^2)(t)\nonumber\\
&-(4b^2+8b^3-32b^5)(\e,z^2)(t)+(\e
[u_{1x}+u_{2x}],z^2)(t)\nonumber\\&+2b(\e, z_y^2)(t)-b(\e(u_1+u_2),z^2)(t)=0.
\label{u4}
\end{align}

We estimate
\begin{align}
&I_1=(\e(u_{1x}+u_{2x}),z^2)(t)\leq
\|u_{1x}+u_{2x}\|(t)\|e^{bx}z\|^2_{L^4(\mathcal{S})}(t)\nonumber\\&\leq
2\|u_{1x}+u_{2x}\|(t)\|e^{bx}z\|(t)\|\nabla(e^{bx}z)\|(t)
\leq
\delta(\e,[2{z_x}^2+{z_y}^2])(t)\nonumber\\&+[2{b^2}\delta+\frac{2}{\delta}(\|u_{1x}\|^2(t)+\|u_{2x}\|^2(t))](\e,z^2)(t),\nonumber\\
&I_2=b(\e(u_{1}+u_{2}),z^2)(t)\leq
b\|u_{1}+u_{2}\|(t)\|e^{bx}z\|^2_{L^4(\mathcal{S})}(t)\nonumber\\&\leq
2b\|u_{1}+u_{2}\|(t)\|e^{bx}z\|(t)\|\nabla(e^{bx}z)\|(t)\nonumber\\
&\leq
\delta(\e,2z_x^2+z_y^2)(t)+[2b^2\delta+\frac{2b^2}{\delta}(\|u_{1}\|^2(t)
+\|u_{2}\|^2(t))](\e,z^2)(t).\nonumber
\end{align}
Substituting $I_1,I_2$ into \eqref{u4} and taking $\delta>0$
sufficiently small, we find

\begin{align}
&\frac{d}{dt}(\e,z^2)(t)\leq
C(b)\big[1+\|u_1\|^2(t)\nonumber\\&+\|u_2\|^2(t)+\|u_{1x}\|^2(t)+\|u_{2x}\|^2(t)\big](\e,z^2)(t).\label{u5}
\end{align}

Since
$$u_i\in L^{\infty}(0,T;L^2(\mathcal{S})),\quad u_{ix}\in
L^2(0,T;L^2(\mathcal{S}))\quad i=1,2,
$$
 then by the Gronwall lemma,
$$(\e, z^2)(t)=0 \quad\forall \:t\in(0,T).$$
Hence, $u_1=u_2 \quad a.e.$ in $\mathcal{S}_T.$
\end{proof}
\begin{rem}
Changing initial condition \eqref{u3} for $z(x,y,0)=z_0(x,y)\ne 0,$
and repeating the proof of Theorem 3.4, we obtain from \eqref{u5}
that
$$(\e,z^2)(t)\leq C(b,T,\|u_0\|)(\e,z^2_0)\quad \forall t\in(0,T).$$
This means continuous dependence of regular solutions on initial
data.
\end{rem}

\section{Decay of regular solutions}\label{regdecay}
 In this section we will prove exponential decay of regular
 solutions in an elevated weighted norm. We start with  Theorem  \ref{decay1} which
 is crucial for the main result.

\begin{teo} \label{tedec1}
Let  $b\in(0,b_0),\;
\|u_0\|\leq \frac{3\pi}{8B}$ and $u(x,y,t)$ be a regular solution of
\eqref{2.1}-\eqref{2.3}. Then for all finite $B>0$ the following
inequalities are true:
\begin{align}
& \|e^{bx}u\|^2(t)\leq e^{-\chi t}\|e^{bx}u_0\|^2(0),\label{dec1}\\
&\int_0^t e^{\chi s}(\e,|\nabla u|^2)(s)\,ds\leq C(b,\|u_0\|)(1+t)(\e,u_0^2), \label{dec2}
\end{align}
 where
 $$\chi=\frac{b_0\pi^2}{4B^2},\quad b_0=\min\big(\frac{\sqrt{0,6}}{2},\quad\frac{1}{5}[-1+\sqrt{1+\frac{5\pi^2}{4B^2}}] \big).$$
\end{teo}
\begin{proof}
Multiplying \eqref{2.1} by $2\e u$, we get the equality
\begin{align}
&\frac{d}{dt}(e^{2bx},u^2)(t)+(2+6b-40b^3)(e^{2bx},u^2_x)(t)+10b(\e,u_{xx}^2)(t)\nonumber\\&+2b(e^{2bx},u^2_y)(t)-\frac{4b}{3}(e^{2bx},u^3)(t)\nonumber\\& -(4b^2+8b^3-32b^5)(e^{2bx},u^2)(t)=0.
\label{ed1}
\end{align}
Taking into account \eqref{p1}, we estimate
\begin{align}
&I=\frac{4b}{3}(e^{2bx},u^3)(t)\leq b(e^{2bx},u_y^2+2u_x^2+2b^2
u^2)(t)\nonumber\\
&+\frac{16b}{9}\|u_0\|^2(e^{2bx},u^2)(t).\nonumber
\end{align}

The following proposition is principal for our proof.
\begin{prop}\label{propcr}
\begin{equation}
\int_{\mathbb{R}}\int_0^B e^{2bx}u^2(x,y,t)\,dy\,dx\leq
\frac{B^2}{\pi^2}\int_{\mathbb{R}}\int_0^B
e^{2bx}u^2_y(x,y,t)\,dy\,dx. \label{mineq}
\end{equation}

\end{prop}
\begin{proof}
Since $u(x,0,t)=u(x,B,t)=0,$ fixing $(x,t)$, we can use with respect
to $y$ the following Steklov inequality: if $f(y)\in H^1_0(0,\pi)$
then
$$\int_0^{\pi}f^2(y)\,dy\leq \int_0^{\pi} |f_y(y)|^2\,dy.$$
After a corresponding process of scaling we prove  Proposition
\ref{propcr}.
\end{proof}

 Making use of \eqref{mineq} and substituting $I$ into \eqref{ed1},
 we come to the following inequality:
\begin{align}
&\frac{d}{dt}(e^{2bx},u^2)(t)+(e^{2bx},u^2_x)(t)\nonumber\\
&+\big[\frac{b\pi^2}{B^2}-
4b^2-10b^3-\frac{16b}{9}\|u_0\|^2\big](\e,u^2)(t)\leq 0 \nonumber
\end{align}
which can be rewritten as

\begin{equation} \frac{d}{dt}(e^{2bx},u^2)(t) +\chi
(e^{2bx},u^2)(t)\leq 0, \label{d2}
\end{equation}
where
$$\chi=b\big[\frac{\pi^2}{B^2}-4b-10b^2-\frac{16\|u_0\|^2}{9}\big].$$
Since we need $\chi>0,$ define
\begin{equation}
4b+10b^2=\gamma\frac{\pi^2}{B^2},\quad
\frac{16\|u_0\|^2}{9}=(1-\gamma)^2\frac{\pi^2}{B^2}, \label{algebra}
\end{equation}
where $\gamma\in(0,1).$ It implies
$\chi=bA(\gamma)\frac{\pi^2}{B^2}$ with
$A(\gamma)=\gamma(1-\gamma).$\\
It is easy to see that
$$ \sup_{\gamma\in(0,1)}A(\gamma)=A(\frac{1}{2})=\frac{1}{4}.$$
Solving \eqref{algebra}, we find
$$b=\frac{1}{5}[-1+\sqrt{1+\frac{5\pi^2}{4B^2}}],\quad \|u_0\|\leq
\frac{3\pi}{8B},\quad \chi=b\frac{\pi^2}{4B^2},$$ and from
\eqref{d2} we get
$$ (e^{2bx},u^2)(t)\leq e^{-\chi
t}(e^{2bx},|u_0|^2).$$ The last inequality
 implies \eqref{dec1}.\par
 \end{proof} To prove \eqref{dec2}, we return to \eqref{ed1} and multiply it by $e^{\chi t}$ to obtain
  \begin{align}
  &\frac{d}{dt}[e^{\chi t}(\e,u^2)(t)]+e^{\chi t}[(2+6b^2-40b^3)(\e,u_x^2)(t)+2b(\e,u_y^2)(t)\nonumber\\&
  +10b(\e,u_{xx}^2)(t)]=\frac{4b e^{\chi t}}{3}(\e,u^3)(t)+\nonumber\\
  &e^{\chi t}(\chi+4b^2+8b^3-32b^5)(\e,u^2)(t). \label{ed2}
 \end{align}
  Acting as above, we find
  \begin{align}
  &I=\frac{4b}{3}(\e,u^3)(t)\leq b(\e,2u_x^2+u_y^2)(t)+10b(\e,u_{xx}^2)(t)\nonumber\\
  &+(2b^2+\frac{16b\|u_0\|^2}{9})(\e,u^2)(t).
  \end{align}
 Substituting this into \eqref{ed2}, we get
 \begin{align}
 &\frac{d}{dt}[e^{\chi t}(\e,u^2)(t)]+e^{\chi t}(\e,|\nabla u|^2)(t)+10be^{\chi t}(\e,u_{xx}^2)(t)\nonumber\\
 &\leq C(b,\|u_0\|)e^{\chi t}(\e,u^2)(t).
 \end{align}
 
 Integrating and  \eqref{dec1} imply
 \begin{align}
 &e^{\chi t}(\e,u^2)(t)+\int_0^t e^{\chi s}(\e,|\nabla u|^2+u_{xx}^2)(s)\,ds\nonumber\\
 &\leq C(b,\|u_0\|)(1+t)(\e,u_0^2). \label{ed3}
 \end{align}
 The proof of Theorem \ref{tedec1} is complete.
 
 Observe that differently from \cite{larkin1,larkinH1,h-strip}, we do not have
 any restrictions on the width of a strip $B$.

The main result of this section is the following assertion.
\begin{teo} \label{tedec2}
Let all the conditions of Theorem \ref{tedec1} be fulfilled. Then
regular solutions of \eqref{2.1}-\eqref{2.3} satisfy the following
inequality:
\begin{align}
&(\e,u^2+|\nabla u|^2+u_{xx}^2)(t)\nonumber\\&\leq C(b,\|u_0\|)(1+t)e^{-\chi t}(\e,
\big[u_0^2+|\nabla u_0|^2+u_{0xx}^2\big]) \label{decaymain}
\end{align}
or
$$\|e^{bx}u\|^2_{H^1(\mathcal{S})}(t)+\|u_{xx}\|^2(t)\leq C(b,\|u_0\|)(1+t)e^{-\chi
t}(\e,
u_0^2 +|\nabla u_0|^2+u_{0xx}^2).$$
\end{teo}
\begin{proof} We start with the following lemma.
\begin{lem} \label{lem1}
Regular solutions of (2.1)-(2.3) satisfy the following equality:
\begin{align}
 &e^{\chi t}(\e, |\nabla u|^2+u_{xx}^2)(t)+2\int_0^t
e^{\chi
s}\{(1+3b-20b^3)(\e,|\nabla u_{x}|^2+u_{xxx}^2)(s)\nonumber\\
&+b(\e,|\nabla u_{y}|^2+u_{xxy}^2)(s)+5b(\e,|\nabla u_{xx}|^2+|\partial_x^4 u|^2)(s)\}\,ds=\nonumber\\
&\int_0^t e^{\chi s}(\chi+4b^2+8b^3-32b^5)(\e,|\nabla
u|^2+u_{xx}^2)(s)\,ds\nonumber\\&+\int_0^t e^{\chi s}\{4b(\e u,u_x^2)(s)
+2(\e u u_x, u_{yy}+4b^2 u_{xx}+4bu_{xxx}+\partial_x^4 u)(s)\nonumber\\
&-(\e u^2,u_{xxx}+2bu_{xx})(s)\}\,ds+(\e,|\nabla u_0|^2+u_{0xx}^2).\label{e412}
\end{align}
\end{lem}
\begin{proof}
First we transform the scalar product
\begin{align}
&-2(e^{2bx}\big[u_t-u_{xx}+u_{xxx}+u_{xyy}+uu_x-\partial_x^5 u\big],\nonumber\\
&\big[u_{yy}+u_{xx}-\partial_x^4 u\big])(t)=0 \label{e2l1}
\end{align}
into the following equality:
\begin{align}
& \frac{d}{dt}(\e, |\nabla u|^2+u_{xx}^2)(t) +2(1+3b-20b^3)(\e,|\nabla u_{x}|^2+u_{xxx}^2)(t)\nonumber\\
& +2b(\e,|\nabla u_{y}|^2+u_{xxy}^2)(t)
+10b(\e,|\nabla u_{xx}|^2+|\partial_x^4 u|^2)(t)\nonumber\\
& =4b^{2}(1+2b-8b^3)(\e,|\nabla u|^2+u_{xx}^2)(t)+2(\e uu_x, u_{yy}\nonumber\\
&+4b^2 u_{xx}+4bu_{xxx}+\partial_x^4 u)(t)-(\e u^2, u_{xxx}+2b u_{xx})(t)\nonumber\\& +4b(\e,uu_x^2)(t). 
\end{align}
Multiplying this by $e^{\chi t}$ and integrating, we prove \eqref{e412}.
\end{proof}

Making use of Lemma \ref{supr}, estimate separate terms in \eqref{e412} as follows:

\begin{align}
& I_1=2(\e uu_x, u_{yy}+4b^2 u_{xx}+4bu_{xxx}+\partial_x^4 u)(t)\nonumber\\
&\leq 2\|u\|(t)\sup_{\mathcal{S}}|e^{bx}u_x|(t)\|e^{bx}[u_{yy}+4b^2 u_{xx}+4bu_{xxx}+\partial_x^4 u]\|(t)\nonumber\\
& \leq \epsilon(1+\|u_0\|^2)(\e,u_{yy}^2+u_{xx}^2+u^2_{xxx}+|\partial_x^4 u|^2)(t)\nonumber\\&+\frac{C(b)}{\epsilon}\{2\delta(\e,u_{xy}^2+u_{xxy}^2)(t) +\frac{3}{\delta}(\e,u_{xx}^2+u_{x}^2)(t)\}.\nonumber\\
& I_2=4b(\e,uu_x^2)(s)\leq \delta(\e,2u_{xx}^2+u_{xy}^2)(s)\nonumber\\
&+[2b^2\delta+\frac{16(1+4b)^2\|u_0\|^2}{\delta}](\e,u_x^2)(s);\nonumber\\
& I_3=(\e u^2,u_{xxx}+2bu_{xx})(t)\leq\|u\|(t)\sup_{\mathcal{S}}|e^{bx}u|(t)\|e^{bx}[u_{xxx}+2bu_{xx}]\|(t)\nonumber\\
&\leq \frac{\epsilon}{2}(\e,u_{xxx}^2+u_{xx}^2)(t)+C(b)\|u_0\|^2[\delta(\e,u_{xy}^2)(t)\nonumber\\&+\frac{1}{\delta}(2\e, |\nabla u|^2+u^2)(t)].\nonumber
\end{align}

Choosing $\epsilon, \delta $ sufficiently small,  substituting $I_1-I_3 $ into  \eqref{e412} and taking into account \eqref{dec1}, we prove that
\begin{align}
& e^{\chi t}(\e,|\nabla u|^2+u_{xx}^2)(t)\leq C(b,\|u_0\|)(1+t)(\e,
u_0^2+|\nabla u_0|^2+u_{0xx}^2).\nonumber
\end{align}
 Adding \eqref{dec1}, we complete
the proof of Theorem \ref{tedec2}.
\end{proof}

\section{Weak solutions}\label{weak}

Here we will prove the existence, uniqueness and continuous
dependence on initial data  as well as exponential decay
 results for weak solutions of \eqref{2.1}-\eqref{2.3} when the initial function $u_0\in L^2(\mathcal{S}).$

\begin{teo}\label{weakexist}
Let $u_0 \in L^2 (\mathcal{S})\cap L^2_b(\mathcal{S}).$ Then for all
finite positive $T$ and $B$ there exists at least one function $u(x,y,t)$:
$$u \in L^{\infty}(0,T;L^2(\mathcal{S})),\;u_x\in
          L^2(0,T;L^2(\mathcal{S}))$$ such that $$e^{bx}u \in
L^{\infty}(0,T;L^2(\mathcal{S}))\cap L^2(0,T;H^1(\mathcal{S})),$$
$$e^{bx}u_{xx}\in L^2(0,T;L^2(\mathcal{S}))$$
 and
the following integral identity takes a place:
\begin{align}\
&(e^{bx} u,v)(T)+\int_0^T\{-(e^{bx} u,v_t)(t)-\frac{1}{2}(\e u^2,bv+v_x)(t)\nonumber\\&+ (e^{bx} u_{xx},[v_{xxx}+3bv_{xx}+(3b^2-1) v_x+(b^3-b-1) v])(t)\nonumber\\
&+(e^{bx} u_y,bv_y+v_{xy})(t)\}\,dt=(e^{bx}
u_0,v(x,y,0)),\label{weakdef}
\end{align}
where $v \in C^{\infty}(\mathcal{S}_T)$ is an arbitrary function.
\end{teo}
\begin{proof}
In order to justify our calculations, we must operate with
sufficiently smooth solutions $u^m(x,y,t)$. With this purpose, we
consider first initial functions $u_{0m}(x,y)$,  which satisfy
conditions of Theorem \ref{regsol}, and obtain estimates \eqref{E1},
\eqref{Ert} for functions $u^m(x,y,t)$. This allows us to pass to
the limit as $m\to \infty$ in the following identity:
\begin{align}
&(e^{bx} u^m,v)(T)+\int_0^T\{-(e^{bx} u^m,v_t)(t)-\frac{1}{2}(\e |u^m|^2,bv+v_x)(t)\nonumber\\&+ (e^{bx} u^m_{xx},[v_{xxx}+3bv_{xx}+(3b^2-1) v_x+(b^3-b-1) v])(t)\nonumber\\
&+(e^{bx} u^m_y,bv_y+v_{xy})(t)\}\,dt=(e^{bx}
u^m_0,v(x,y,0)),
\end{align}
and come to \eqref{weakdef}.
\end{proof}

{\bf  Uniqueness of a weak solution.}

\begin{teo} \label{weakuniq}
A weak solution of Theorem \ref{weakexist} is uniquely defined.
\end{teo}

\begin{proof}
Actually, this proof is provided by Theorem \ref{uniq}. It is
sufficient to approximate the initial function $u_0\in
L^2(\mathcal{S})$ by regular functions $u_{0m}$ in the form:
$$\lim_{m\to\infty}\|u_{0m}-u_0\|=0,$$
where $u_{om}$ satisfies the conditions of Theorem \ref{regsol}.
This guarantees the existence of the unique regular solution to
(2.1)-(2.3) and allows us to repeat all the calculations which have
been done during the proof of Theorem \ref{uniq} and to come to the
following inequality:
\begin{align}
&\frac{d}{dt}(\e,z_m^2)(t)+(\e,| \nabla z_{m}|^2)(t)\nonumber\\
&\leq
C(b)\big[1+\|u_{1m}\|^2(t)+\|u_{2m}\|^2(t)+\|u_{1xm}\|^2(t)+\|u_{2xm}\|^2(t)\big](\e,z_m^2)(t).\nonumber
\end{align}
By the generalized Gronwall`s lemma,
\begin{align}
&(\e,z_m^2)(t)\leq exp\{\int_0^t
C(b)\big[1+\|u_{1m}\|^2(s)+\|u_{2m}\|^2(s)+\|u_{1xm}\|^2(s)\nonumber\\
&+\|u_{2xm}\|^2(s)\big]\,ds\}(\e,z_{0m}^2)(t).\nonumber
\end{align}
Functions $u_{1m}$ and $u_{2m}$ for $m$ sufficiently large satisfy
the estimate
\begin{equation}
\|u_{im}\|^2(t)+2\int_0^t \|u_{imx}\|^2(s)\,ds=\|u_{0m}\|^2\leq
2\|u_0\|^2),\quad i=1,2. \nonumber
\end{equation}

Hence,
\begin{align}
&exp\{\int_0^t
C(b)\big[1+\|u_{1m}\|^2(s)+\|u_{2m}\|^2(s)+\|u_{1xm}\|^2(s)\nonumber\\
&+\|u_{2xm}\|^2(s)\big]\,ds\}\leq C(,T,\|u_0\|).
\end{align}
 Since $e^{bx}z(x,y,t)$ is a weak limit of regular
solutions $\{e^{bx}z_m(x,y,t)\}$, then $$(\e,z^2)(t)\leq (\e,
z_m^2)(t)= 0.$$ This implies $u_1\equiv u_2$\; $a.e.$ in
$\mathcal{S}_T.$ The proof of Theorem \ref{weakuniq} is complete.
\end{proof}
\begin{rem}
Changing initial condition $z(x,y,0)\equiv 0$  for
$z(x,y,0)=z_0(x,y)\ne 0,$ and repeating the proof of Theorem
\ref{weakuniq}, we obtain  that
$$(\e,z^2)(t)\leq C(b,T,\|u_0\|)(\e,z^2_0)\quad \forall t\in(0,T).$$
This means continuous dependence of weak solutions on initial data.
\end{rem}

{\bf  Decay of weak solutions.}

\begin{teo} \label{decay3}
Let  $b\in(0,b_0),\;
\|u_0\|\leq \frac{3\pi}{16B}$ and $u(x,y,t)$ be a regular solution of
\eqref{2.1}-\eqref{2.3}. Then for all finite $B>0$ the following
inequality is true:
\begin{equation}
 \|e^{bx}u\|^2(t)\leq e^{-\chi t}\|e^{bx}u_0\|^2(0),
\label{ë5}
\end{equation}
 where
 $$\chi=\frac{b_0\pi^2}{4B^2},\quad b_0=\min\big(\frac{\sqrt{0,6}}{2},\quad\frac{1}{5}[-1+\sqrt{1+\frac{5\pi^2}{4B^2}}] \big).$$
\end{teo}

\begin{proof}
Similarly to the proof of the uniqueness result for a weak solution,
we approximate $u_0 \in L^2(\mathcal{S})$ by sufficiently smooth
functions $u_{om}$ in order to work with regular solutions. Acting
in the same manner as by the proof of Theorem 4.1, we come to the
following inequality :
\begin{equation}
 \|e^{bx}u_m\|^2(t)\leq e^{-\chi t}\|e^{bx}u_0\|^2(0),
\label{wdecay}
\end{equation}
 where
 $$\chi=\frac{\pi^2}{20B^2}[-1+\sqrt{1+\frac{5\pi^2}{4B^2}}].$$
Since $u(x,y,t)$ is weak limit of regular solutions $\{u_m(x,y,t)\}$
then
$$(\e, u^2)(t)\leq (\e, u_m^2)(t)\leq e^{-\chi t}(\e,u_0^2).$$
The proof of Theorem \ref{decay3} is complete.
\end{proof}

We have in this Theorem a more strict condition  $\|u_0\|\leq
\frac{3\pi}{16B}$ instead of $\|u_0\|\leq \frac{3\pi}{8B}$ in the
case of decay for regular solution because  for weak solutions we do
not have the sharp estimate \eqref{E1l}, but only \eqref{E1}.

\end{document}